\documentclass[a4paper,12pt]{article}
\usepackage{a4wide}
\usepackage{amsmath}
\usepackage{amssymb}
\usepackage{amsthm}
\usepackage{latexsym}
\usepackage{graphicx}
\usepackage[english]{babel}
\usepackage{makeidx}

\newtheorem{obs} [subsection]{Remark}
\newtheorem{exm} [subsection]{Example}

\newtheorem{prop}[subsection]{Proposition}
\newtheorem{conj}[subsection]{Conjecture}
\newtheorem{teor}[subsection]{Theorem}
\newtheorem*{teor*}{Theorem}
\newtheorem{lema}[subsection]{Lemma}
\newtheorem{cor} [subsection]{Corollary}

\newcommand{\pa}{p_{\mathbf a}}

\newcommand{\za}{\zeta_{\mathbf a}}

\def\dx{\operatorname{dx}}
\def\dy{\operatorname{dy}}
\def\dt{\operatorname{dt}}
\def\Ree{\operatorname{Re}}

\begin{document}
\selectlanguage{english}
\frenchspacing
\numberwithin{equation}{section}

\large
\begin{center}
\textbf{On the restricted partition function via determinants with Bernoulli polynomials}

Mircea Cimpoea\c s
\end{center}
\normalsize

\begin{abstract}
Let $r\geq 1$ be an integer, $\mathbf a=(a_1,\ldots,a_r)$ a vector of positive integers and let $D\geq 1$ be a common multiple of $a_1,\ldots,a_r$.
We prove that, if a determinant $\Delta_{r,D}$, which depends only on $r$ and $D$, with entries consisting in values of Bernoulli polynomials is nonzero, 
then the restricted partition function $\pa(n): = $ the number of integer solutions $(x_1,\dots,x_r)$ to
$\sum_{j=1}^r a_jx_j=n$ with $x_1\geq 0, \ldots, x_r\geq 0$ can be computed in terms of values of Bernoulli polynomials and Bernoulli Barnes numbers.
 

\end{abstract}

\section{Introduction}

Let $\mathbf a:=(a_1,a_2,\ldots,a_r)$ be a sequence of positive integers, $r\geq 1$. 
The \emph{restricted partition function} associated to $\mathbf a$ is $\pa:\mathbb N \rightarrow \mathbb N$, 
$\pa(n):=$ the number of integer solutions $(x_1,\ldots,x_r)$ of $\sum_{i=1}^r a_ix_i=n$ with $x_i\geq 0$.
Let $D$ be a common multiple of $a_1,\ldots,a_r$. According to \cite{bell}, $\pa(n)$ is a 
quasi-polynomial of degree $r-1$, with the period $D$, i.e.
$$\pa (n) = d_{\mathbf a,r-1}(n)n^{r-1} + \cdots + d_{\mathbf a,1}(n)n + d_{\mathbf a,0}(n),\;\text{ for all }n\geq 0,$$
where $d_{\mathbf a,m}(n+D)=d_{\mathbf a,m}(n)$, $\text{ for all }0\leq m\leq r-1,n\geq 0$, and $d_{\mathbf a,r-1}(n)$ is not identically zero.
The restricted partition function $\pa(n)$ was studied extensively in the literature, starting with the works of Sylvester \cite{sylvester} and Bell \cite{bell}.
Popoviciu \cite{popoviciu} gave a precise formula for $r=2$. Recently, Bayad and Beck \cite[Theorem 3.1]{babeck} proved an
explicit expression of $\pa(n)$ in terms of Bernoulli-Barnes polynomials and the Fourier Dedekind sums, in the case that
$a_1,\ldots,a_r$ are are pairwise coprime.  
In our paper, we propose a new approach to compute $\pa(n)$ using the methods introduced in 
 \cite{lucrare}. In formula \eqref{poola} we prove that
$$ \sum_{m=0}^{r-1}\sum_{v=1}^{D} d_{\mathbf a,m}(v)D^{n+m} \frac{B_{n+m+1}(\frac{v}{D})}{n+m+1}  = \frac{(-1)^{r-1} n!}{(n+r)!}B_{r+n}(\mathbf a)-\delta_{0n},\text{ for all }n\in\mathbb N,$$
where $B_{j}(x)$'s are the \emph{Bernoulli polynomials}, $B_{j}(\mathbf a)$'s are the \emph{Bernoulli-Barnes numbers} and $\delta_{0n}$ is the \emph{Kronecker symbol}. Taking $0\leq n\leq rD-1$ and seeing $d_{\mathbf a,m}(v)$'s
as indeterminates, we obtain a system of $rD$ linear equations with the determinant denoted by $\Delta_{r,D}$.
In Proposition $2.1$, we note that if $\Delta_{r,D}\neq 0$, then $d_{\mathbf a,m}(v)$'s
can be computed using the Cramer's rule. We conjecture that $\Delta_{r,D}\neq 0$ for any $r,D\geq 1$.
We define
$$\Phi_j(x):=b_1\frac{B_j(x)}{j} + b_2\frac{B_{j+1}(x)}{j+1}+\cdots+b_{rD}\frac{B_{rD+j-1}(x)}{rD+j-1},\;1\leq j\leq r,$$
where $b_1,\ldots,b_{rD}\in\mathbb R$. In Proposition $2.3$, we prove that $\Delta_{r,D}\neq 0$ if and only if
$$\Phi_j(0)=\Phi_j(\frac{1}{D})=\cdots=\Phi_j(\frac{D-1}{D})=0, \text{ for all }1\leq j\leq r, \text{ implies } b_1=\cdots=b_{rD}=0.$$
In Proposition $2.6$ we prove that the above condition holds for $r=1$, hence $\Delta_{1,D}\neq 0$.

In Theorem $2.11$, we prove that if $\Phi_j(0)=\Phi_j(\frac{1}{D})=\cdots=\Phi_j(\frac{D-1}{D})=0, \text{ for all }1\leq j\leq r$, then
there exists a polynomial $Q\in\mathbb R[x]$ of degree $\leq r-1$ such that 
$$\Phi_j(x)=A_{j,0}(x)S(x) + A_{j,1}(x)S(x+1) + \cdots + A_{j,r-1}(x)S(x+r-1),$$
where $S(x)=(Dx)_{Dr}Q(x)$, $(Dx)_{Dr}=Dx(Dx-1)\cdots(Dx-Dr+1)$ and  
$$A_{j,\ell}(x)=(-1)^{r-1}\sum_{t=0}^{\ell}(-1)^t \binom{r}{t} (x+\ell-t)^{j-1},\; \text{ for all }1\leq j\leq r,\; 0\leq \ell\leq r-1.$$
Using this, in Corollary $2.12$ we prove that $\Delta_{2,D}\neq 0$. Writing 
$$Q(x)=a_0+a_1(x-r)+a_2(x-r)(x-\frac{r}{2})+\cdots+a_{r-1}(x-r)(x-\frac{r}{2})^{r-2},$$
we define $\Phi_{j,t}(x)=A_{j,0}(x)S_t(x) + A_{j,1}(x)S_t(x+1) + \cdots + A_{j,r-1}(x)S_t(x+r-1)$, where $S_0(x)=1$, $S_t(x)=(x-r)(x-\frac{r}{2})^{t-1}$ for $t\geq 1$.
We let $I_{j,t}:=\int_0^1\Phi_{j,t}(x)\dx$. In Corollary $2.14$ we prove that 
$\Delta_{r,D}\neq 0$  if and only if $\Delta'_{r,D}:=\det(I_{j,t})_{\substack{1\leq j\leq r \\ 0\leq t\leq r-1}}\neq 0$. In Proposition $2.16$
we prove that $I_{j,t}=0$ for $t+(D+1)t+j\equiv 1(\bmod\; 2)$. 

In the second section, we consider the case $D=1$ and $r\geq 1$, i.e. $\mathbf a=(1,\ldots,1)$. Using the arithmetics of the Bernoulli numbers, we 
show in Theorem $3.4$ that $\Delta_{1,r}\neq 0$. 

\section{Main results}

Let $\mathbf a:=(a_1,a_2,\ldots,a_r)$ be a sequence of positive integers, $r\geq 1$. 
The \emph{restricted partition function} associated to $\mathbf a$ is $\pa:\mathbb N \rightarrow \mathbb N$, 
$$\pa(n):= \#\{ (x_1,\ldots,x_r)\in\mathbb N^r\;:\; \sum_{i=1}^r a_ix_i=n\},\;\text{ for all }n\geq 0.$$
Let $D$ be a common multiple of $a_1,\ldots,a_r$. Bell \cite{bell} has proved that $\pa(n)$ is a 
quasi-polynomial of degree $r-1$, with the period $D$, i.e.
\begin{equation}
\pa (n) = d_{\mathbf a,r-1}(n)n^{r-1} + \cdots + d_{\mathbf a,1}(n)n + d_{\mathbf a,0}(n),\;\text{ for all }n\geq 0,
\end{equation}
where $d_{\mathbf a,m}(n+D)=d_{\mathbf a,m}(n)$, for all $0\leq m\leq r-1,n\geq 0$, and $d_{\mathbf a,r-1}(n)$ is not identically zero.
The \emph{Barnes zeta} function associated to $\mathbf a$ and $w>0$ is
$$ \za(s,w):=\sum_{n=0}^{\infty} \frac{\pa(n)}{(n+w)^s},\; \Ree s>r,$$
see \cite{barnes} and \cite{spreafico} for further details. It is well known that $\za(s,w)$ is meromorphic on $\mathbb C$ with poles at most in the set $\{1.\ldots,r\}$.
We consider the function
\begin{equation}
\za(s) := \lim_{w\searrow 0}(\za(s,w)-w^{-s}).
\end{equation}
In \cite[Lemma 2.6]{lucrare} it was proved that 
\begin{equation}
\za(s)=\frac{1}{D^s}\sum_{m=0}^{r-1}\sum_{v=1}^{D} d_{\mathbf a,m}(v)D^m \zeta(s-m,\frac{v}{D}), 
\end{equation}
where $$\zeta(s,w):=\sum_{n=0}^{\infty} \frac{1}{(n+w)^s},\;\Ree s>1,$$
is the \emph{Hurwitz zeta} function. 
The \emph{Bernoulli numbers} $B_j$ are defined by
$$ \frac{z}{e^z-1} = \sum_{j=0}^{\infty}B_j \frac{z^j}{j!}, $$
$B_0=1$, $B_1=-\frac{1}{2}$, $B_2=\frac{1}{6}$, $B_4=-\frac{1}{30}$ and $B_n=0$ if $n$ is odd and greater than $1$.
The \emph{Bernoulli polynomials} are defined by
$$\frac{ze^{xz}}{(e^z-1)}=\sum_{n=0}^{\infty}B_n(x)\frac{z^n}{n!}.$$
They are related with the Bernoulli numbers by
$$B_n(x)=\sum_{k=0}^n \binom{n}{k}B_{n-k}x^k.$$
The Bernoulli polynomials satisfy the identities:
\begin{equation}
\int_x^{x+1} B_n(t) \dt = x^n,\;\text{ for all }n\geq 1.
\end{equation}
\begin{equation}
B_{n}(x+1)-B_n(x) = nx^{n-1},\;\text{ for all }n\geq 1.
\end{equation}
It is well know, see for instance \cite[Theorem 12.13]{apostol}, that
\begin{equation}
\zeta(-n,w)=-\frac{B_{n+1}(w)}{n+1},\;\text{ for all }n\in \mathbb N,w>0. 
\end{equation}
The \emph{Bernoulli-Barnes polynomials} are defined by
$$\frac{z^r e^{xz}}{(e^{a_1 z}-1)\cdots (e^{a_r z}-1)}=\sum_{j=0}^{\infty}B_j(x;\mathbf a)\frac{z^j}{j!}. $$
The \emph{Bernoulli-Barnes numbers} are defined by
$$B_j(\mathbf a):=B_j(0;\mathbf a)=\sum_{i_1+\cdots+ i_r=j}\binom{j}{i_1,\ldots,i_r} B_{i_1}\cdots B_{i_r}a_1^{i_1-1}\cdots a_r^{i_r-1}.$$
According to \cite[Formula (3.10)]{rui}, we have that
\begin{equation}
\za(-n,w)=\frac{(-1)^r n!}{(n+r)!}B_{r+n}(w;\mathbf a),\;\text{ for all }n\in\mathbb N. 
\end{equation}
From $(2.2)$ and $(2.7)$ it follows that
\begin{equation}
\za(-n)=\frac{(-1)^r n!}{(n+r)!}B_{r+n}(\mathbf a),\;\text{ for all }n\geq 1,\; \za(0)=\frac{(-1)^r}{r!}B_{r}(\mathbf a)-1.
\end{equation}
From $(2.3),(2.6)$ and $(2.8)$ it follows that
\begin{equation}\label{poola}
\sum_{m=0}^{r-1}\sum_{v=1}^{D} d_{\mathbf a,m}(v)D^{n+m} \frac{B_{n+m+1}(\frac{v}{D})}{n+m+1}  = \frac{(-1)^{r-1} n!}{(n+r)!}B_{r+n}(\mathbf a)-\delta_{0n},\;\text{ for all }n\in\mathbb N,
\end{equation}
where $\delta_{0n}=\begin{cases} 1,& n=0,\\ 0,& n\geq 1\end{cases}$ is the \emph{Kronecker's symbol}.
Given values $0\leq n\leq rD-1$ in $(2.9)$ and seeing $d_{\mathbf a,m}(v)$'s as indeterminates, we obtain a system of linear equations with the determinant \small{
\begin{equation}\label{pista}
\Delta_{r,D}:= \begin{vmatrix}
\frac{B_1(\frac{1}{D})}{1} & \cdots & \frac{B_1(1)}{1} 
& \cdots & D^{r-1}\frac{B_r(\frac{1}{D})}{r} & \cdots & D^{r-1}\frac{B_r(1)}{r} \\
D\frac{B_2(\frac{1}{D})}{2} & \cdots & D\frac{B_2(1)}{2} 
& \cdots & D^{r}\frac{B_{r+1}(\frac{1}{D})}{r+1} & \cdots & D^{r}\frac{B_{r+1}(1)}{r+1} \\
\vdots & \vdots & \vdots & \vdots & \vdots & \vdots & \vdots  \\
D^{rD-1}\frac{B_{rD}(\frac{1}{D})}{rD} & \cdots & D^{rD-1}\frac{B_{rD}(1)}{rD} & 
 \cdots & D^{rD+r-2}\frac{B_{rD+r-1}(\frac{1}{D})}{rD+r-1} & \cdots & D^{rD+r-2}\frac{B_{rD+r-1}(1)}{rD+r-1} 
\end{vmatrix} \end{equation}}

\begin{prop}
With the above notations, if $\Delta_{r,D}\neq 0$, then 
$$d_{\mathbf a,m}(v) = \frac{\Delta_{r,D}^{m,v}}{\Delta_{r,D}},\;\text{ for all } 1\leq v\leq D, 0\leq m\leq r-1,$$ 
where $\Delta_{r,D}^{m,v}$ is the determinant obtained from $\Delta_{r,D}$, as defined in \eqref{pista}, by replacing the 
$(mD+v)$-th column with the column $(\frac{(-1)^{r-1} n!}{(n+r)!}B_{n+r}(\mathbf a)-\delta_{n0})_{0\leq n\leq rD-1}$.
\end{prop}

\begin{proof}
It follows from \eqref{poola} and \eqref{pista} using the Cramer's rule.
\end{proof}

\begin{cor}
 With the above notations, if $\Delta_{r,D}\neq 0$, then 
$$\pa(n)=\frac{1}{\Delta_{r,D}}\sum_{m=0}^{r-1} \Delta_{r,D}^{m,v} n^m,\;\text{ for all }n\in\mathbb N.$$
\end{cor}

\begin{proof}
 If follows from Proposition $2.1$ and $(2.1)$.
\end{proof}

Let $r\geq 1$ be an integer and $b_1,b_2,\ldots,b_{rD}$ be some numbers. For $1\leq j\leq r$, let 
\begin{equation}\label{coor}
\Phi_j(x):=b_1\frac{B_j(x)}{j} + b_2\frac{B_{j+1}(x)}{j+1}+\cdots+b_{rD}\frac{B_{rD+j-1}(x)}{rD+j-1}\in \mathbb Q[x].
\end{equation}

\begin{prop}
 The following are equivalent:
\begin{enumerate}
 \item[(1)] $\Delta_{r,D}\neq 0$.
 \item[(2)] $\Phi_j(0)=\Phi_j(\frac{1}{D})=\cdots=\Phi_j(\frac{D-1}{D})=0,\;\text{ for all }1\leq j\leq r$, implies $b_1=b_2=\cdots=b_{rD}=0$.
\end{enumerate}
\end{prop}

\begin{proof}
Since $B_n(1-x)=(-1)^n B_n(x),\;\text{ for all }x\in \mathbb R,n\geq 0$, from \eqref{pista} it follows that
$$
\Delta_{r,D} = D^{\frac{rD(rD +r-2)}{2}} \begin{vmatrix}
\frac{B_1(\frac{1}{D})}{1} & \cdots & \frac{B_1(1)}{1}  
& \cdots & \frac{B_r(\frac{1}{D})}{r} & \cdots & \frac{B_r(1)}{r} \\
\frac{B_2(\frac{1}{D})}{2} & \cdots & \frac{B_2(1)}{2}  
& \cdots & \frac{B_{r+1}(\frac{1}{D})}{r+1} & \cdots & \frac{B_{r+1}(1)}{r+1} \\
\vdots & \vdots & \vdots & \vdots & \vdots & \vdots & \vdots  
\\
\frac{B_{rD}(\frac{1}{D})}{rD} & \cdots & \frac{B_{rD}(1)}{rD} & 
 \cdots & \frac{B_{rD+r-1}(\frac{1}{D})}{rD+r-1} & \cdots & \frac{B_{rD+r-1}(1)}{rD+r-1} 
\end{vmatrix} =
$$
\begin{equation}\label{pista2}
 = (-1)^{\frac{rD(rD +r)}{2}} D^{\frac{rD(rD +r-2)}{2}} \begin{vmatrix}
\frac{B_1(\frac{D-1}{D})}{1} & \cdots & \frac{B_1(0)}{1}  
& \cdots & \frac{B_r(\frac{D-1}{D})}{r} & \cdots & \frac{B_r(0)}{r} \\
\frac{B_2(\frac{D-1}{D})}{2} & \cdots & \frac{B_2(0)}{2}  
& \cdots & \frac{B_{r+1}(\frac{D-1}{D})}{r+1} & \cdots & \frac{B_{r+1}(0)}{r+1} \\
\vdots & \vdots & \vdots & \vdots & \vdots & \vdots & \vdots  
\\
\frac{B_{rD}(\frac{D-1}{D})}{rD} & \cdots & \frac{B_{rD}(0)}{rD} & 
 \cdots & \frac{B_{rD+r-1}(\frac{D-1}{D})}{rD+r-1} & \cdots & \frac{B_{rD+r-1}(0)}{rD+r-1} 
\end{vmatrix}.
\end{equation}
Evaluating \eqref{coor} in $x=0,\frac{1}{D},\ldots,\frac{D-1}{D}$ and seeing $b_1,\ldots,b_{rD}$ as variables, from \eqref{pista2}
we obtain a system of linear equations with the determinant $C\cdot \Delta_{r,D}$, where $C\neq 0$. By Cramer's rule, we get the required result.
\end{proof}

Our computer experiments \cite{DGPS} yield us to propose the following conjecture.

\begin{conj}
$\Delta_{r,D}\neq 0$, for any integers $r,D\geq 1$.
\end{conj}

We denote the \emph{falling factorial} (also known as the \emph{Pochhammer symbol}), 
$$(x)_D = x(x-1)\cdots (x-D+1).$$

\begin{lema}
For any $D\geq 1$ we have that $\int_0^D (x)_D dx > 0$.
\end{lema}

\begin{proof}
If $D=2k+1$ then
\begin{equation}\label{l41}
\int_0^{D-1} (x)_D \dx = \int_{-k}^k (x+k)_D \dx =  \int_{-k}^k x(x^2-1)\cdots (x^2-(k-1)^2)\dx = 0.
\end{equation}
From \eqref{l41} it follows that 
$$ \int_0^D (x)_D \dx = \int_{D-1}^D (x)_D \dx > 0.$$
If $D=2k$ then
\begin{equation}\label{l42}
\int_0^{D} (x)_D \dx 
 = \frac{2}{D+1}\int_{0}^D (x)_D(x-\frac{D}{2})\dx -\frac{2}{D+1}\int_{0}^D (x)_{D+1}\dx.
\end{equation}
From \eqref{l41} and \eqref{l42} it follows that
\begin{equation}\label{l43}
\int_0^{D} (x)_D \dx =\frac{2}{D+1}\int_{0}^D (x)_D(x-\frac{D-1}{2})\dx > \frac{2}{D+1}\int_{0}^{D-1} (x)_D(x-\frac{D-1}{2})\dx.
\end{equation}
On the other hand, \small{
$$\int_{0}^{D-1} (x)_D(x-\frac{D-1}{2})\dx = \int_{-\frac{D-1}{2}}^{\frac{D-1}{2}}x(x+\frac{D-1}{2})_D\dx =
\int_{-\frac{D-1}{2}}^{\frac{D-1}{2}} x (x^2-\frac{1}{4})\cdots (x^2-\frac{(D-1)^2}{4})\dx = 0,$$}
hence, from \eqref{l42} and \eqref{l43} it follows that
$$ \int_0^D (x)_D \dx = \int_{D-1}^D (x)_D \dx > 0,$$
thus the proof is complete.
\end{proof}

\begin{prop}
Assume $r=1$ and let $\Phi(x):=\Phi_1(x)=b_1B_1(x)+b_2\frac{B_2(x)}{2}+\cdots+b_D \frac{B_D(x)}{D}$.

If $\Phi(0)=\Phi(\frac{1}{D})=\cdots=\Phi(\frac{D-1}{D})=0$ then $\Phi=0$.
\end{prop}

\begin{proof}
We let $F(x):=\Phi(\frac{x}{D})$. The hypothesis implies $F(x)=c(x)_D$ for some $c\in\mathbb R$.
On the other hand, from $(2.4)$ we have that
$$0 = \int_0^1 \Phi(x)\dx = \frac{1}{D}\int_0^D F(x)\dx = \frac{c}{D}\int_0^D (x)_D\dx,$$
therefore, by Lemma $2.5$, it follows that $c=0$ and thus $\Phi=0$.
\end{proof}

\begin{cor}
For any $D\geq 1$, we have that $\Delta_{1,D}\neq 0$.
\end{cor}

\begin{proof}
 It follows from Proposition $2.3$ and Proposition $2.6$.
\end{proof}

Let $n\geq 1$ be an integer. The forward difference operator $\Delta_n$ applied to a function $F(x)$ is
$$\Delta_n F(x):=F(x+n)-F(x).\;\text{We denote } \Delta:=\Delta_1.$$

\begin{lema}
Let $F_j(x):=D^{j-1}\Phi_j(\frac{x}{D})$, $1\leq j\leq r$. Assume that
$$F_j(0)=F_j(1)=\cdots=F_j(D-1)=0,\;\text{ for all }1\leq j\leq r.$$
There exists a polynomial $P\in\mathbb R[X]$ of degree $\leq r-1$ such that 
$$F_1(x)=\Delta_D^{r-1}((x)_{Dr}P(x)) \text{ and } \Delta_D F_j(x) = x^{j-1} \Delta_D F_1(x),\; 2\leq j\leq r.$$
\end{lema}

\begin{proof}
From $(2.5)$ and $(2.11)$ it follows that
\begin{equation}\label{curul}
\Delta_D F_j(x):= b_1x^{j-1}+\frac{b_2}{D}x^j+\cdots+\frac{b_{Dr}}{D^{Dr-1}}x^{Dr+j-2} = x^{j-1}\Delta_D F_1(x),\;\text{ for all }1\leq j\leq r.
\end{equation}
Let $1\leq j\leq r$. The hypothesis implies $(x)_D|F_j(x)$. We write $F_j(x)=(x)_D P_j(x)$, $1\leq j\leq r$, where $P_j\in\mathbb R[X]$
is a polynomial of degree $\leq (r-1)D+j-1$.

For $1\leq j \leq r$ and $0\leq t\leq r-j$, we define inductively the polynomials $P^t_j\in\mathbb R[X]$ of degree $\leq (r-t)D+j+t-1$ with
the property 
\begin{equation}\label{curulet}
F_j(x)=\Delta_D^{t} ((x)_{D(t+1)}P^{t}_j(x)), 
\end{equation}
where $\Delta_D^t$ is the $t$-th power of the operator $\Delta_D$. 
We let $P^0_j:=P_j$, $1\leq j\leq r$. Let $1 \leq t\leq r-2$ and $j\leq r-t-1$. Assume that \eqref{curulet} holds for $j$, $j+1$ and $t$, i.e.
\begin{equation}\label{curuletu}
F_{j}(x)=\Delta_D^{t-1} ((x)_{Dt}P^{t-1}_{j}(x)) =
 \sum_{\ell=0}^{t-1} \binom{t-1}{\ell}(-1)^{t-1-\ell}(x+\ell D)_{Dt}P^{t-1}_j(x+\ell D) .
\end{equation}
\begin{equation}\label{curuletu2}
F_{j+1}(x)=\Delta_D^{t-1} ((x)_{Dt}P^{t-1}_{j+1}(x)) =
 \sum_{\ell=0}^{t-1} \binom{t-1}{\ell}(-1)^{t-1-\ell}(x+\ell D)_{Dt}P^{t-1}_{j+1}(x+\ell D) .
\end{equation}
Since, by \eqref{curul}, $x\Delta_D(F_j(x))=\Delta_D(F_{j+1}(x))$, from \eqref{curuletu} and \eqref{curuletu2} it follows that
$$ \sum_{\ell=0}^{t-1} \binom{t-1}{\ell}(-1)^{t-1-\ell} [ x(x+(\ell+1) D)_{Dt}P^{t-1}_j(x+(\ell+1) D) - x(x+\ell D)_{Dt}P^{t-1}_j(x+\ell D) - $$
\begin{equation}\label{curr}
- (x+(\ell+1) D)_{Dt}P^{t-1}_{j+1}(x+(\ell+1) D) + (x+\ell D)_{Dt}P^{t-1}_{j+1}(x+\ell D)]=0.
\end{equation}
Evaluating \eqref{curr} in $x=0,1,\ldots, D-1$, it follows that
\begin{equation}\label{l81}
P^{t-1}_{j+1}(tD+\ell)=\ell P^{t-1}_{j}(tD+\ell), \; 0\leq \ell\leq D-1.
\end{equation}
Hence, there exists a polynomial $P^{t}_j\in \mathbb R[X]$ of degree $\leq (r-t)D+j+t-1$ such that
\begin{equation}\label{l82}
(x-tD)P^{t-1}_{j}(x) - P^{t-1}_{j+1}(x) = tD(x-tD)_D P^{t}_{j}(x).
\end{equation}
From \eqref{curul}, \eqref{curuletu}, \eqref{curuletu2} and \eqref{l82} it follows that 
$$0 = \Delta_D F_{j+1}x- x\Delta_D F_j(x) =  \Delta_D^{t} ((x)_{Dt}P^{t-1}_{j+1}(x)) - x \Delta_D^{t} ((x)_{Dt}P^{t-1}_{j}(x)) = $$
\begin{equation}\label{l88}
 = tD \Delta_D^t ( (x)_{D(t+1)}P^{t}_j(x) ) + \Delta_D^t ( (x-tD)P^{t-1}_{j}(x) ) - x \Delta_D^{t} ((x)_{Dt}P^{t-1}_{j}(x)).
\end{equation}
On the other hand, one can easily check that
\begin{equation}\label{l89}
x \Delta_D^t((x)_{Dt}P^{t-1}_{j}(x)) - \Delta_D^t ( (x-tD)(x)_{Dt}P^{t-1}_{j}(x) ) = tD \Delta_D^{t-1}((x)_{Dt}P^{t-1}_{j}(x)) = tD F_j(x),
\end{equation}
From \eqref{l88} and \eqref{l89}, it follows that
$F_j(x)=\Delta_D^t ( (x)_{D(t+1)}P^{t}_j(x) ),$ 
hence the induction step is complete. 

Finally, $P(x):=P^{r-1}_1(x)$, which is a polynomial of degree $\leq r-1$, satisfies the conclusion.
\end{proof}

In the following example, we explicit the proof of Lemma $2.8$ for $r=1,2$:

\begin{exm}\emph{
 (1) Assume $r=1$ and $D\geq 1$. We have that 
     $$F_1(x)=\Phi_1(\frac{x}{D}) = b_1 B_1(\frac{x}{D}) + \frac{b_2}{2}B_2(\frac{x}{D}) + \cdots +\frac{b_D}{D}B_D(\frac{x}{D}).$$
     From the assumption of Lemma $2.8$, that is $F_1(0)=\ldots=F_D(D-1)=0$, it follows that there exists a constant $c\in \mathbb Q$ such that
     $F_1(x)=c(x)_D$, in other words $P(x)=c$.}

\emph{
  (2) Assume $r=1$ and $D\geq 1$. We have that 
      \begin{align*}
      & F_1(x)=b_1 B_1(\frac{x}{D}) + \frac{b_2}{2}B_2(\frac{x}{D}) + \cdots +\frac{b_{2D}}{2D}B_{2D}(\frac{x}{D}),\\
      & F_2(x)=D\left(\frac{b_1}{2} B_2(\frac{x}{D}) + \frac{b_2}{3}B_3(\frac{x}{D}) + \cdots +\frac{b_{2D}}{2D+1}B_{2D+1}(\frac{x}{D})\right).\\
      \end{align*}
      Hence, from $(2.5)$ it follows that 
      \begin{equation}\label{pasat}
      \Delta_D F_2(x) =  b_1 x + b_2 \frac{x^2}{D} + \cdots + b_{2D} \frac{x^{2D}}{D^{2D-1}} = x \Delta_D F_1(x).
      \end{equation}
      On the other hand, from the hypothesis of Lemma $2.8$, we can write
      \begin{equation}\label{pass}
      F_1(x)=(x)_D P_1(x),\;F_2(x)=(x)_D P_2(x),
      \end{equation}
      where $P_1$ and $P_2$ are two polynomials of degrees $\leq D$ and $\leq D+1$, respectively. Hence, from \eqref{pasat} and \eqref{pass}, we get
      \begin{equation}\label{pasata}
      (x+D)_D P_2(x+D) - (x)_D P_2(x) = x((x+D)_D P_1(x+D) - (x)_D P_1(x)).
      \end{equation}
      Evaluating \eqref{pasata} in $x=0,1,\ldots,D-1$, it follows that 
      \begin{equation}\label{pasataa}
       P_2(\ell+D)=\ell P_1(\ell+D)\text{ for }0\leq \ell \leq D-1.
      \end{equation}
      From \eqref{pasataa} it follows that there exists a polynomial $P$ of degree $\leq 1$ such that
      \begin{equation}\label{pasatt}
      (x-D)P_1(x) - P_2(x) = D(x-D)_D P(x).
      \end{equation}
      From \eqref{pass} and \eqref{pasatt}, it follows that 
      \begin{equation}\label{passatt}
      (x-D)F_1(x)-F_2(x) = D(x)_{2D} P(x).
      \end{equation}
      Applying $\Delta_D$ in \eqref{passatt}, from \eqref{pasat} it follows that
      $$\Delta_D( (x-D)F_1(x)-F_2(x) ) = x\Delta_D F_1(x) + DF_1(x) - \Delta_D F_2(x) = DF_1(x)=D\Delta_D((x)_{2D}P(x)),$$
      therefore $F_1(x)=\Delta_D((x)_{2D}P(x))$, as required.}
\end{exm}

\begin{lema}
Let $S:[0,r]\rightarrow \mathbb R$ be a function and let $1\leq j\leq r$. Then:
$$ x^{j-1}\Delta^r S(x) = \Delta(A_{j,0}(x) S(x) + A_{j,1}(x) S(x+1) + \cdots + A_{j,r-1}(x) S(x+r-1)),$$
where $A_{j,\ell}$'s are polynomials of degree $j-1$, defined by
$$A_{j,\ell}(x)=(-1)^{r-1}\sum_{t=0}^{\ell}(-1)^t \binom{r}{t} (x+\ell-t)^{j-1} = (-1)^{r+j}A_{j,r-1-\ell}(1-x),\;0\leq \ell\leq r-1.$$
\end{lema}

\begin{proof}
We have that
\begin{equation}\label{l51}
\Delta^{r} S(x) = \sum_{\ell=0}^{r} (-1)^{r-\ell}\binom{r}{\ell}S(x+\ell).
\end{equation}
On the other hand, 
\begin{equation}\label{l52}
\Delta(A_{j,\ell}(x)S(x+\ell)) = A_{j,\ell}(x+1)S(x+\ell+1)-A_{j,\ell}(x)S(x+\ell).
\end{equation}
From hypothesis, \eqref{l51} and \eqref{l52} it follows that
\begin{equation}\label{l53}
 \sum_{\ell=0}^{r} (-1)^{r-\ell}\binom{r}{\ell}S(x+\ell)x^{j-1} = \sum_{\ell=0}^{r-1} (A_{j,\ell}(x+1)S(x+\ell+1)-A_{j,\ell}(x)S(x+\ell)).
\end{equation}
The relation \eqref{l53} is satisfied for {\small
\begin{equation}\label{l54}
A_{j,\ell-1}(x+1)-A_{j,\ell}(x) = (-1)^{r-\ell}\binom{r}{\ell}x^{j-1},1\leq \ell\leq r-1,\;A_{j,0}(x)= (-1)^{r-1}x^{j-1}
,A_{j,r-1}(x+1)=x^{j-1}.
\end{equation}}
From \eqref{l54} it follows, inductively, that 
\begin{equation}\label{l55}
A_{j,\ell}(x)= A_{j,\ell-1}(x+1)+(-1)^{r-\ell-1}\binom{r}{\ell}x^{j-1} = \sum_{t=0}^{\ell}(-1)^{r-t-1}\binom{r}{t}(x+\ell-t)^{j-1},\;
1\leq \ell \leq r-1.
\end{equation}
We prove that the polynomials
\begin{equation}\label{l56}
\Psi_{j}(x):=- \sum_{t=0}^{r}(-1)^{r-t}\binom{r}{t}(x+r-t)^{j-1},\;1\leq j\leq r,
\end{equation}
are zero, using induction on $j\geq 1$. 
For $j=1$, we have $$\Psi_1(x)=- \sum_{t=0}^{r}(-1)^{r-t}\binom{r}{t} = -(1-1)^t = 0.$$
Now, assume $2\leq j\leq r$. By induction hypothesis, it follows that
$$\Psi'_j(x) = (j-1) \Psi_{j-1}(x)=0,\;\text{ for all }x\in\mathbb R,$$
hence $\Psi_j$ is constant. On the other hand,
$$\Psi_j(0)= - \sum_{t=0}^r (-1)^{r-t}\binom{r}{t} (r-t)^{j-1} = - \sum_{t=0}^r (-1)^t \binom{r}{t} t^{j-1} $$
For any $0\leq t\leq r$, we write 
$$t^{j-1}=c_{j-1}(t)_{j-1}+\cdots + c_2(t)_2+c_1t+c_0,\;$$
where $c_k$'s are uniquely determined by $j$. It follows that
$$\Psi_j(0) = - \sum_{t=0}^r (-1)^t \binom{r}{t} \sum_{k=0}^{j-1} c_k(t)_{k} = - \sum_{k=0}^{j-1} c_k \sum_{t=0}^r (-1)^t \binom{r}{t}(t)_{k}
= $$
$$ = - \sum_{k=0}^{j-1} c_k(r)_{k} \sum_{t=0}^r (-1)^t \binom{r-k}{t-k}  \sum_{k=0}^{j-1} (-1)^{k-1} c_k(r)_{k} = $$
$$ = \sum_{\ell=0}^{r-k} (-1)^{\ell} \binom{r-k}{\ell} \sum_{k=0}^{j-1} (-1)^{k-1} c_k(r)_{k} (1-1)^{r-k} = 0,$$
hence $\Psi_j=0$ as required. In particular, we have
$$A_{j,r-1}(x) = \Psi_j(x)+x^{j-1} = x^{j-1},$$
hence \eqref{l54} is satisfied. Moreover, since $\Psi_j=0$, from \eqref{l55} and \eqref{l56} it follows that
$$A_{j,\ell}(x) = - \sum_{t=\ell+1}^{r}(-1)^{r-t-1}\binom{r}{t}(x+\ell-t)^{j-1},\;\text{ hence } A_{j,r-1-\ell}(1-x) =$$
\begin{equation}\label{l57}
  = - \sum_{t=r-\ell}^{r}(-1)^{r-t-1}\binom{r}{t}(1-x+r-1-\ell-t)^{j-1} = 
(-1)^j \sum_{t=r-\ell}^{r}(-1)^{r-t-1}\binom{r}{t} (x+\ell-r+t)^{j-1}.
\end{equation}
Substituting $k=r-t$ in \eqref{l57}, it follows that
$$A_{j,r-1-\ell}(1-x) = (-1)^j \sum_{k=0}^{\ell}(-1)^{k-1}\binom{r}{r-k}(x+\ell-k)^{j-1} = $$
$$ = (-1)^{j+r} \sum_{k=0}^{\ell}(-1)^{r-k-1}\binom{r}{k}(x+\ell-k)^{j-1} = (-1)^{j+r} A_{j,\ell}(x),$$
as required.
\end{proof}

\begin{teor}
With the notations from $(2.11)$, assume that
$$\Phi_j(0)=\Phi_j(\frac{1}{D})=\cdots=\Phi_j(\frac{D-1}{D})=0,\; \text{ for all }1\leq j\leq r.$$
There exists a polynomial $Q\in\mathbb R[x]$ of degree $\leq r-1$ such that 
$$\Phi_j(x)=A_{j,0}(x)S(x) + A_{j,1}(x)S(x+1) + \cdots + A_{j,r-1}(x)S(x+r-1),$$
where $S(x)=(Dx)_{Dr}Q(x)$ and 
$$A_{j,\ell}(x)=(-1)^{r-1}\sum_{t=0}^{\ell}(-1)^t \binom{r}{t} (x+\ell-t)^{j-1},\; \text{ for all }1\leq j\leq r, 0\leq \ell\leq r-1.$$
\end{teor}

\begin{proof}
From Lemma $2.8$ it follows that there exists a polynomial $P\in \mathbb R[x]$ of degree $\leq r-1$ such that
\begin{equation}\label{t1}
\Phi_1(x)=F_1(Dx) = \Delta_D^{r-1} (Dx)_{rD}P(Dx).
\end{equation}
We let $$Q(x):=\frac{1}{D^{Dr}}P(Dx) \text{ and } S(x):=(Dx)_{Dr}Q(x).$$ 
From \eqref{t1} it follows that 
$\Phi_1(x)=\Delta^{r-1} S(x)$. Moreover, according to \eqref{curul}, we have
\begin{equation}\label{t2}
\Delta \Phi_j(x) = \frac{1}{D^{j-1}}\Delta_D F_j(Dx) =x^{j-1} \Delta_D F_1(Dx) = x^{j-1}\Delta \Phi_1(x),\;\text{ for all }2\leq j\leq r.
\end{equation}
From \eqref{t1}, \eqref{t2} and Lemma $2.10$ it follows that
$$\Delta \Phi_j(x) = \Delta (A_{j,0}(x)S(x)+\cdots+A_{j,r-1}(x)S(x+r-1)),\;\text{ for all }1\leq j\leq r.$$
Since $\Phi_j(0)=0$, $\text{ for all }1\leq j\leq r$, and $S(0)=\cdots=S(r-1)=0$, it follows that
$$\Phi_j(x)=A_{j,0}(x)S(x) + A_{j,1}(x)S(x+1) + \cdots + A_{j,r-1}(x)S(x+r-1),$$
as required.
\end{proof}

\begin{cor}
If $r=2$ and $$\Phi_j(0)=\Phi_j(\frac{1}{D})=\cdots=\Phi_j(\frac{D-1}{D})=0,\; j=1,2$$
then $\Phi_1=\Phi_2=0$. Consequently, $\Delta_{2,D}\neq 0$.
\end{cor}

\begin{proof}
From (2.4) it follows that $\int_0^1 \Phi_1(x)\dx = \int_0^1 \Phi_2(x)\dx=0$.
According to Theorem $2.11$ we have
$$\Phi_1(x)=\Delta((Dx)_{2D}Q(x)) = (Dx+D)_{2D}Q(x+1)-(Dx)_{2D}Q(x),$$ 
$$\Phi_2(x)= (x-1)(Dx+D)_{2D}Q(x+1) - x(Dx)_{2D}Q(x),$$
where $Q(x)=ax+b$, with $a,b\in\mathbb R$. It follows that
\begin{equation}\label{c61}
\int_{0}^1 \Phi_2(x) dx = \int_{0}^1 (x-1)(Dx+D)_{2D}Q(x+1) \dx - \int_0^1 x(Dx)_{2D}Q(x)dx.
\end{equation}
Using the substitution $y=1-x$  and the identity 
$(D-Dy)_{2D} = (Dy+D-1)_{2D}$, we get
\begin{equation}\label{c62}
 \int_0^1 x(Dx)_{2D}Q(x)dx = - \int_0^1 (y-1)(Dy+D-1)_{2D}Q(1-y)dy. 
\end{equation}
From \eqref{c61} and \eqref{c62} it follows that 
$$\int_{0}^{1} \Phi_2(x) \dx = \int_0^1 (Dx+D-1)_{2D}((x+1)Q(x+1)+(x-1)Q(1-x))\dx = $$
\begin{equation}\label{c63}
= (4a+2b)\int_0^1 x(Dx+D-1)_{2D}dx = 0.
\end{equation}
Using a similar argument as in the formula $(2.13)$, we get $$\int_0^1 x(Dx+D-1)_{2D}dx = 0,$$ hence $2a+b=0 \Rightarrow Q(x)=a(x-2)$.
It follows that
\begin{equation}\label{c64}
 \int_0^1 \Phi_1(x) dx = a \int_{0}^1 ((Dx+D)_{2D}(x-1)-(Dx)_{2D}(x-2)) \dx
\end{equation}
Using the substitution $y=1-x$ , we get
$$\int_0^1 (Dx)_{2D}(x-2) \dx = - \int_0^1 (Dy+D-1)_{2D}(y-1)dy,$$
hence, by \eqref{c64},
$$ \int_0^1 \Phi_1(x) \dx = 2a \int_0^1 x(Dx+D-1)_{2D}dx =0, $$
which implies, as above, that $a=0$. Therefore $b=0$, hence $\Phi_1=\Phi_2=0$. The last assertion follows from Proposition $2.3$.
\end{proof}

Assume $r\geq 2$. Let $Q\in\mathbb R[x]$ be the polynomial from the statement of Theorem $2.11$. We write
\begin{equation}\label{cucu1}
Q(x)=a_0+a_1(x-r)+a_2(x-r)(x-\frac{r}{2})+\cdots+a_{r-1}(x-r)(x-\frac{r}{2})^{r-2}.
\end{equation}
We let 
\begin{equation}\label{cucu2}
S_t(x):=\begin{cases} (x)_{Dr},& t=0\\ (x-r)(x-\frac{r}{2})^{t-1}(x)_{Dr}, &1\leq t\leq r-1  \end{cases}.
\end{equation}
We define 
\begin{equation}\label{cucu3}
\Phi_{j,t}(x):=A_{j,0}(x)S_t(x)+A_{j,1}(x)S_t(x+1)+\cdots+A_{j,r-1}(x)S_t(x+r-1),\;1\leq j\leq r,0\leq t\leq r-1.
\end{equation}
From \eqref{cucu1}, \eqref{cucu2}, \eqref{cucu3} and Theorem $2.11$ it follows that
\begin{equation}\label{cucu4}
\Phi_j(x)=a_0 \Phi_{j,0}(x)+ a_1 \Phi_{j,1}(x) + \cdots +a_{r-1} \Phi_{j,r-1}(x),\;\text{ for all }1\leq j\leq r.
\end{equation}
We let $$I_{j,t}:=\int_0^1 \Phi_{j,t}(x)\dx,\;1\leq j\leq r,0\leq t\leq r-1.$$
From (2.4) and \eqref{cucu4} it follows that
\begin{equation}\label{cucu5}
\int_0^1 \Phi_j(x)\dx = a_0 I_{j,0} + a_1 I_{j,1} + \cdots + a_{r-1} I_{j,r-1}=0,\;\text{ for all }1\leq j\leq r.
\end{equation}
With the above notations, we have the following result.

\begin{cor}
$\Delta'_{r,D}:=\begin{vmatrix}
I_{1,0} & I_{1,1} & \cdots & I_{1,r-1} \\ 
I_{2,0} & I_{2,1} & \cdots & I_{2,r-1} \\ 
\vdots & \vdots & \ddots  & \vdots \\ 
I_{r,0} & I_{r,1} & \cdots & I_{r,r-1}
\end{vmatrix} \neq 0 $ if and only if $\Phi_1=\Phi_2=\cdots=\Phi_r=0$.
\end{cor}

\begin{proof}
From \eqref{cucu5} and the Cramer's rule it follows that the system of linear equations
$$ a_0 I_{j,0} + a_1 I_{j,1}+\cdots + a_{r-1} I_{j,r-1}=0,\;1\leq j\leq r,$$
has the unique solution $a_0=a_1=\cdots=a_{r-1}=0$ if and only if $\Delta'_{r,D}\neq 0$.
On the other hand, $a_0=a_1=\cdots=a_{r-1}=0 \Leftrightarrow \Phi_1=\Phi_2=\cdots=\Phi_r=0$.
\end{proof}

\begin{cor}
Let $r,D\geq 1$ be two integers. We have that: $\Delta_{r,D}\neq 0 \Leftrightarrow \Delta'_{r,D}\neq 0$.
Consequently, Conjecture $2.4$ is equivalent to: $\Delta'_{r,D}\neq 0$, for any integers $r,D\geq 1$.
\end{cor}

\begin{proof}
It follows from Proposition $2.3$ and Corollary $2.13$.
\end{proof}

\begin{obs}
\emph{Although $\Delta'_{r,D}$ is a $r\times r$-determinant, while $\Delta_{r,D}$ is a $rD \times rD$-determinant, the
computation of $\Delta'_{r,D}$ seems difficult and we were unable to prove that $\Delta'_{3,D}\neq 0$ for arbitrary $D\geq 1$.
A small step in the general case is given in the following proposition.}
\end{obs}

\begin{prop}
Let $r,D\geq 1$ be two integers. For any $1\leq t\leq r-1$ and $1\leq j\leq r$ such that $t+(D+1)t+j\equiv 1(\bmod 2)$, we have 
$$I_{j,t}=\int_{0}^1 \Phi_{j,t}(x)\dx = 0.$$
\end{prop}

\begin{proof}
We fix $1\leq t\leq r-1$ and $1\leq j\leq r$. According to \eqref{cucu2} and \eqref{cucu3}, we have that
\begin{equation}\label{p1}
\Phi_{j,t}(x)=A_{j,0}(x)(x-r)(x-\frac{r}{2})^{t-1}(Dx)_{Dr} + 
\cdots + A_{j,r-1}(x)(x-1)(x+\frac{r}{2}-1)^{t-1}(Dx+D(r-1))_{Dr}.
\end{equation}
On the other hand, from Lemma $2.9$ it follows that
\begin{equation}\label{p2}
A_{j,\ell}(1-x)=(-1)^{r+j}A_{j,r-1-\ell}(x),\;\text{ for all }0\leq \ell\leq r-1.
\end{equation}
Let $0\leq \ell\leq r-1$. Using the substitution $y=1-x$ and \eqref{p2} it follows that
$$ \int_0^1 A_{j,\ell}(x)(x-r+\ell)(x-\frac{r}{2}+\ell)^{t-1}(Dx+\ell D)_{Dr}\dx = $$
$$ = \int_1^0 A_{j,\ell}(1-y)(1-y-r-\ell)(1-y-\frac{r}{2}+\ell)^{t-1}(D-Dy+\ell D)_{Dr}(-1)\dy = $$
$$ =(-1)^{t+Dr} \int_0^1 A_{j,\ell}(1-y)(y+r-1-\ell)(y-\frac{r}{2}+r-\ell-1)^{t-1}(D(y+r-1-\ell)-1)_{Dr}\dy = $$
\begin{equation}\label{p3}
 =(-1)^{t+(D+1)r+j} \int_0^1 A_{j,r-1-\ell}(y)(y-1-\ell)(y-\frac{r}{2}+r-\ell-1)^{t-1}(D(y+r-1-\ell))_{Dr}\dy 
\end{equation}
From hypothesis, \eqref{p1} and \eqref{p3} it follows that $I_{j,t}=0$.
\end{proof}

\section{The case $D=1$}

Let $r\geq 1$ be an integer. Using the notations from the first section, by $(2.12)$, we have
\begin{equation}
 \Delta_{r,1} = (-1)^{r}\begin{vmatrix}
  B_1 & \frac{1}{2}B_2 & \cdots  & \frac{1}{r}B_r  \\ 
\frac{1}{2}B_2 & \frac{1}{3}B_3 & \cdots  & \frac{1}{r+1}B_{r+1}  \\ 
\vdots &  \vdots &  \vdots &  \vdots   \\
\frac{1}{r}B_r & \frac{1}{r+1}B_{r+1}  & \cdots & \frac{1}{2r-1}B_{2r-1}
\end{vmatrix}.
\end{equation}

The following lemma is an easy exercise of linear algebra.

\begin{lema} For any $k\geq 1$, we have that:
$$(1) \;\begin{vmatrix}
x_1 & x_2 & 0 & x_4 &  \cdots & 0 & x_{2k}  \\ 
x_2 & 0   & x_4 & 0 &  \cdots & x_{2k} & 0 \\ 
0  & x_4 & 0 & x_6 &  \cdots & 0 & x_{2k+2} \\
\vdots &  \vdots &  \vdots &  \vdots &  \vdots &  \vdots  &  \vdots   \\
x_{2k} & 0 & x_{2k+2} & \cdots & \cdots & 0 & x_{4k-2}
\end{vmatrix} = 
(-1)^k
\begin{vmatrix}
x_2 & x_4 & \cdots  & x_{2k}  \\ 
x_4 & x_6  & \cdots & x_{2k+2} \\ 
\vdots &  \vdots &  \vdots &  \vdots   \\
x_{2k} & x_{2k+2}  & \cdots & x_{4k-2}
\end{vmatrix}^2. $$
$$(2)\; \begin{vmatrix}
x_1 & x_2 & 0 & x_4 &  \cdots & x_{2k} & 0 \\ 
x_2 & 0   & x_4 & 0 &  \cdots & 0 & x_{2k+2} \\ 
0  & x_4 & 0 & x_6 &  \cdots & x_{2k+2} & 0 \\
\vdots &  \vdots &  \vdots &  \vdots &  \vdots &  \vdots  &  \vdots   \\
x_{2k+2} & 0 & x_{2k+4} & \cdots & \cdots & x_{4k} & 0
\end{vmatrix} = 
(-1)^kx_1
\begin{vmatrix}
x_4 & x_6 & \cdots  & x_{2k+2}  \\ 
x_6 & x_8  & \cdots & x_{2k+4} \\ 
\vdots &  \vdots &  \vdots &  \vdots   \\
x_{2k+2} & x_{2k+4}  & \cdots & x_{4k}
\end{vmatrix}^2. \Box$$
\end{lema}

For any integer $n\geq 1$, let $S_n$ be the \emph{symmetric group} of order $n$. Given $\sigma\in S_n$ a permutation, the \emph{signature} of $\sigma$ is denoted by
$$\varepsilon(\sigma):=\prod_{1\leq i < j\leq n}\frac{\sigma(j)-\sigma(i)}{j-i} \in \{\pm 1\}.$$

\begin{lema}
For any $k\geq 1$, we have that:
$$(1) \; \Delta_{2k,1} = \frac{(-1)^k}{2^k} \left( \sum_{\sigma\in S_k} \epsilon(\sigma) \frac{B_{2\sigma(1)} B_{2\sigma(2)+2} 
\cdots B_{2\sigma(k)+2k-2}}{\sigma(1)(\sigma(2)+1)\cdots(\sigma(k)+k-1)} \right)^2.$$
$$(2) \; \Delta_{2k+1,1} = \frac{(-1)^{k+1}}{2^{k+1}} \left( \sum_{\sigma\in S_k} \epsilon(\sigma) \frac{B_{2\sigma(1)+2} B_{2\sigma(2)+4} 
\cdots B_{2\sigma(k)+2k} }{(\sigma(1)+1)(\sigma(2)+2)\cdots (\sigma(k)+k)} \right)^2.$$
\end{lema}

\begin{proof}
Since $B_{2p+1}=0$ for all $p\geq 1$, the conclusion follows from $(3.1)$, Lemma $3.1$ and the definition of a determinant.
\end{proof}

Let $q=\frac{m}{n}\in \mathbb Q$ with $m,n\in\mathbb Z, n\neq 0$. The $2$-valuation of $q$ is
\begin{equation}\label{hop}
v_2(q):=v_2(m)-v_2(n),\;\text{ where } v_2(n):=\max\{k\in\mathbb N\;:\;2^k|n\}, \; v_2(0)=-\infty.
\end{equation}
According to the von Staudt - Clausen Theorem (see \cite{clausen},\cite{staudt}), we have that
\begin{equation}\label{hopa}
\text{Denominator of } B_{2n}=\prod_{p-1|2n}p, \text{ where }p>0 \text{ are primes}.
\end{equation}
From \eqref{hop} and \eqref{hopa} it follows that 
\begin{equation}\label{hopaa}
v_2(B_{2n})=-1, \text{ for all }n\geq 1. 
\end{equation}
For any integer $k\geq 1$, we consider the maps $\Phi_k,\Psi_k:S_k\rightarrow \mathbb N$, 
$$\Phi_k(\sigma):=\sum_{j=1}^k v_2(j+\sigma(j)-1),\; \Psi_k(\sigma):=\sum_{j=1}^k v_2(j+\sigma(j)),\;\text{ for all }\sigma\in S_k.$$ 
With these notations, we have the following result.

\begin{lema}
(1) $\Phi_k$ has an unique maximal element $\sigma_k\in S_k$, defined recursively as
$$ \sigma_k(j):= \begin{cases} \sigma_{m-k}(j),\; j\leq m-k \\ m+1-j,\; j\geq m-k+1 \end{cases}, \; \sigma_1(1)=1,$$
where $t=\left\lfloor \log_2(k-1) \right\rfloor + 1$ and $m=2^t$.

(2) $\Psi_k$ has an unique maximal element $\tau_k\in S_k$, defined recursively as
$$ \tau_k(j):= \begin{cases} \tau_{m-k}(j),\; j\leq m-k \\ m+1-j,\; j\geq m-k+1 \end{cases}, \; \tau_1(1)=1,$$
where $t=\left\lfloor \log_2(k) \right\rfloor + 1$ and $m=2^t-1$.
\end{lema}

\begin{proof}
(1) Let $\sigma\in S_k$. We write $\sigma=C_1\cdots C_m$ as product of disjoint cycles. One can easily see that 
$$\Phi_k(\sigma)=\Phi_k(C_1)+\cdots+\Phi_k(C_m).$$
Let $C=(i_1i_2\ldots i_r)$ be a cycle of length $r\geq 3$. We have 
$$\Phi_k(C)=v_2(i_1+i_2-1)+\cdots+v_2(i_{r-1}+i_r-1) + v_2(i_r + i_1-1).$$

a) If $r$ is odd, then, without any loss of generality, we may assume that $i_1$ and $i_r$ have the same parity and hence
$v_2(i_r + i_1-1)=0$. If 
$$ v_2(i_1+i_2-1)+v_2(i_3+i_4-1)+\cdots + v_2(i_{r-2}+i_{r-1}-1)\leq v_2(i_2+i_3-1)+v_2(i_4+i_5-1)+\cdots + v_2(i_{r-1}+i_{r}-1),$$
then $\Phi_k((i_1i_2)(i_3i_4)\cdots (i_{r-2}i_{r-1}))\geq \Phi(C)$. Else, $\Phi_k((i_2i_3)(i_4i_5)\cdots (i_{r-1}i_{r}))\geq \Phi(C)$.

b) If $r$ is even, then, if
$$ v_2(i_1+i_2-1)+v_2(i_3+i_4-1)+\cdots + v_2(i_{r-1}+i_{r}-1)\leq v_2(i_2+i_3-1)+v_2(i_4+i_5-1)+\cdots + v_2(i_{r}+i_{1}-1),$$
then $\Phi_k((i_1i_2)(i_3i_4)\cdots (i_{r-1}i_{r}))\geq \Phi(C)$. Else, $\Phi_k((i_2i_3)(i_4i_5)\cdots (i_{r}i_{1}))\geq \Phi(C)$.

Let $\sigma\in S_k$ such that $\sigma$ is maximal for $\Phi_k$. From the previous considerations, we may assume that $\sigma$
is an involution, i.e. $\sigma=\tau_1\tau_2\cdots\tau_p$, where $\tau_j$'s are disjoint transpositions and 
$p=\left\lfloor \frac{k}{2}\right\rfloor$.

Let $u,v\in \{1,\ldots,k\}$ such that $u+v-1=2^t$. We claim that there exists an index $j$ such that $\tau_j=(uv)$.
Assume this is not the case. Without any loss of generality, we may assume $\tau_1 = (uv')$ and $\tau_2=(u'v)$ for some
other indexes $u',v'$. Let $t'=v_2(u'+v'-1)$.

a) It $t'<t$, then $u'+v'-1=2^t$, since $u'+v'-1<2^{t+1}$. It follows that $u+v'-1 \neq 2^t$, hence $v_2(u+v'-1)<t$. Thus
$v_2(u'+v-1)+v_2(u+v'-1)<t+t'$. 

b) If $t'<t$, then $u'+v'-1=2^{t'}\alpha$, where $\alpha$ is odd. Therefore 
$$ (u'+v-1)+(u+v'-1) = (u+v-1)+(u'+v'-1) = 2^t + 2^{t'}\alpha = 2^{t'}(2^{t-t'}+\alpha). $$
This implies $\min\{v_2(u'+v-1), v_2(u+v'-1) \}\leq t'$. Since $v_2(u'+v-1), v_2(u+v'-1)<t$, we get
$v_2(u'+v-1)+v_2(u+v'-1)<t+t'$.

In both cases, we get $\Phi_k((u'v')(uv)\tau_3\cdots \tau_p)>\Phi_k(\sigma)$, a contradiction, and thus we proved the claim.
The claim implies $\sigma(j)=\sigma_k(j)$, for any $j\geq m-k+1$. If $k=m$, we are done. If $k<m$, since $\sigma$ is a 
maximal element for $\Phi_k$, it follows that $\sigma|_{\{1,\ldots,m-k\}}$ is a maximal element for $S_{m-k}$. Using induction on $k\geq 1$,
it follows that $\sigma(j)=\sigma_{m-k}(j)$ for any $1\leq j\leq m-k$, and thus $\sigma=\sigma_k$.

$(2)$ The proof is similar to the proof of $(1)$.
\end{proof}

The main result of this section is the following theorem.

\begin{teor}
For any $r\geq 1$, $\Delta_{1,r}\neq 0$.
\end{teor}

\begin{proof}
It is enough to show that $v_2(\Delta_{r,1})>-\infty$. Assume $r=2k$. Let $\sigma_k\in S_k$ be the permutation given in Lemma $3.3(1)$. 
For any $\sigma_k\neq \sigma\in S_k$, from $(3.4)$ and Lemma $3.3(1)$ it follows that
$$-\infty<v_2\left(\frac{B_{2\sigma_k(1)} B_{2\sigma_k(2)+2} \cdots B_{2\sigma_k(k)+2k-2}}{\sigma_k(1)(\sigma_k(2)+1)\cdots(\sigma_k(k)+k-1)}\right) < 
v_2 \left(\frac{B_{2\sigma(1)} B_{2\sigma(2)+2} \cdots B_{2\sigma(k)+2k-2}}{\sigma(1)(\sigma(2)+1)\cdots(\sigma(k)+k-1)}\right),$$
hence, by Lemma $3.2(1)$, $\Delta_{2k,1}\neq 0$. The proof of the case $r=2k+1$ is similar, using Lemma $3.2(2)$ and Lemma $3.3(2)$.
\end{proof}

\begin{obs}(The case $D=2$) \emph{
Let $r\geq 1$ be an integer. Using the notations from the first section, by $(2.12)$, we have
\begin{equation}\label{crook}
 \Delta_{r,2} = A\cdot \begin{vmatrix}
  B_1(\frac{1}{2}) & B_1 & \frac{1}{2}B_2(\frac{1}{2}) & \frac{1}{2}B_2 & \cdots  & \frac{1}{r}B_r(\frac{1}{2}) & \frac{1}{r}B_r \\ 
\frac{1}{2}B_2(\frac{1}{2}) & \frac{1}{2}B_2 & \frac{1}{3}B_3(\frac{1}{2}) & \frac{1}{3}B_3 &  \cdots & \frac{1}{r+1}B_{r+1}(\frac{1}{2}) & \frac{1}{r+1}B_{r+1}  \\ 
\vdots &  \vdots &  \vdots &  \vdots  & \vdots &  \vdots &  \vdots \\
\frac{1}{2r}B_{2r}(\frac{1}{2}) & \frac{1}{2r}B_{2r} & \frac{1}{2r+1}B_{2r+1}(\frac{1}{2}) & \frac{1}{2r+1}B_{2r+1} &  \cdots & \frac{1}{3r-1}B_{3r-1}(\frac{1}{2}) & \frac{1}{3r-1}B_{3r-1}  \\ 
\end{vmatrix},
\end{equation}
where $A\neq 0$. Using the identity $B_n(\frac{1}{2})=(\frac{1}{2^{n-1}}-1)B_n$ in \eqref{crook}, it follows that \small
\begin{equation}\label{crook2}
 \Delta_{r,2} = A\cdot \begin{vmatrix}
  B_1 & B_1 & \frac{1}{2}B_2 & \frac{1}{2}B_2 & \cdots  & \frac{1}{r}B_r & \frac{1}{r}B_r \\ 
\frac{1}{2\cdot 2}B_2 & \frac{1}{2}B_2 & \frac{1}{2\cdot 3}B_3 & \frac{1}{3}B_3 &  \cdots & \frac{1}{2\cdot(r+1)}B_{r+1} & \frac{1}{r+1}B_{r+1}  \\ 
\vdots &  \vdots &  \vdots &  \vdots  & \vdots &  \vdots &  \vdots \\
\frac{1}{2^{2r-1}\cdot 2r}B_{2r} & \frac{1}{2r}B_{2r} & \frac{1}{2^{2r-1}\cdot(2r+1)}B_{2r+1} & \frac{1}{2r+1}B_{2r+1} &  \cdots & \frac{1}{2^{2r-1}\cdot(3r-1)}B_{3r-1} & \frac{1}{3r-1}B_{3r-1}  \\ 
\end{vmatrix}.
\end{equation}}\normalsize

\emph{
Since $v_2(B_{2n})=-1$ for all $n\geq 1$, $B_1=\frac{1}{2}$ and $B_{2n+1}=0$ for all $n\geq 1$,
from \eqref{crook2} it follows that in order to prove that $\Delta_{2,r}\neq 0$ it is enough to show that the following determinant is nonzero:
\begin{equation}
 \bar \Delta_r:= \begin{vmatrix}
                 1 & \frac{1}{2} & 0 & \frac{1}{4} & \cdots & 1 & \frac{1}{2} & 0 & \frac{1}{4} & \cdots  \\
                 \frac{1}{2} & 0 & \frac{1}{4} & 0 & \cdots & \frac{1}{2\cdot 2} & 0 & \frac{1}{2\cdot 4} & 0 & \cdots  \\
                  0 & \frac{1}{4} & 0 & \frac{1}{6} & \cdots & 0 & \frac{1}{2^2\cdot 4} & 0 & \frac{1}{2^2\cdot 6} & \cdots  \\
                 \vdots & \vdots &\vdots &\vdots &\vdots &\vdots &\vdots &\vdots &\vdots &\vdots \\
                  \frac{1}{2r} & 0 & \frac{1}{2r+2} & 0 & \cdots & \frac{1}{2^{2r-1}\cdot 2r} & 0 & \frac{1}{2^{2r-1}\cdot(2r+2)} & 0 & \cdots
                \end{vmatrix}
\end{equation}
In order to prove that, it is enough to show that $v_2(\bar \Delta_{r})>-\infty$. We denote by $a_{ij}$ the entry on the row $i$ and the column $j$ in $\bar \Delta_r$, i.e. $a_{11}=1$, $a_{12}=\frac{1}{2}$ etc.
For any permutation $\sigma:\{1,2,\ldots,2r\}\to\{1,2,\ldots,2r\}$, we denote $$a_{\sigma}:=\prod_{i=1}^{2r}a_{i\sigma(i)}.$$
Our aim is to construct a permutation
$\bar \sigma:\{1,2,\ldots,2r\}\to\{1,2,\ldots,2r\}$ such that 
$$v_2(\bar \Delta_{r}) = v_2(a_{\bar \sigma})\text{ and }
v_2(\bar \Delta_{r})< v_2(a_{\bar \tau}),$$ for any permutation $\bar \tau:\{1,2,\ldots,2r\}\to\{1,2,\ldots,2r\}$ with $\bar\tau \neq \bar\sigma$ and $a_{\bar \tau}\neq 0$.
Moreover, we want $\bar\sigma$ to satisfy the condition $\bar\sigma({\{1,2\ldots,r\}}) = \{1,2,\ldots,r\}$. 
As in the proof of Lemma $3.3$, we can construct a permutation
$\delta:\{1,2,\ldots,r\}\to \{1,2,\ldots,r\}$ such that $$v_2(\prod_{j=1}^r a_{j\delta(j)}) < v_2(\prod_{j=1}^r a_{j\varepsilon(j)}),$$ for any permutation 
$\varepsilon:\{1,2,\ldots,r\}\to \{1,2,\ldots,r\}$ such that $\varepsilon\neq \delta$ and $\prod_{j=1}^r a_{j\varepsilon(j)}\neq 0$.
We let $\bar\sigma(j)=\delta(j)$ for $1\leq j\leq r$.}

\emph{
Similarly, we can construct a permutation $\delta':\{r+1,\ldots,2r\}\to \{r+1,\ldots,2r\}$ such that 
$$v_2(\prod_{j=r+1}^{2r} a_{j\delta'(j)}) < v_2(\prod_{j=r+1}^{2r} a_{j\varepsilon'(j)}),$$ for any permutation $\varepsilon'$ on
$\{r+1,r+2,\ldots,2r\}$ such that $\varepsilon'\neq \delta'$ and $\prod_{j=r+1}^{2r} a_{j\varepsilon(j)}\neq 0$. We let $\bar\sigma(j)=\delta'(j)$ for $r+1\leq j\leq 2r$.
Now, let $\bar \tau \neq \bar \sigma$ be a permutation on $\{1,2,\ldots,2r\}$ such that $a_{\bar\tau}\neq 0$. 
If $\bar\tau(\{1,2,\ldots,r\})=\{1,2,\ldots,r\}$ then $v_2(a_{\bar\sigma})<v_2(a_{\bar\tau})$ by the
choosing of $\delta$ and $\delta'$. Unfortunately, we are not able to tackle the case $\bar\tau(\{1,2,\ldots,r\})\neq \{1,2,\ldots,r\}$.
However, we believe that our method works, as it does for small values of $r$.}
\end{obs}

\vspace{10pt}
\noindent
\textbf{Aknowledgment}: I would like to express my gratitude to Florin Nicolae for the valuable discussions regarding this paper.

{}

\vspace{2mm} \noindent {\footnotesize
\begin{minipage}[b]{15cm}
Mircea Cimpoea\c s, Simion Stoilow Institute of Mathematics, Research unit 5, P.O.Box 1-764,\\
Bucharest 014700, Romania, E-mail: mircea.cimpoeas@imar.ro
\end{minipage}}
\end{document}